\documentclass{amsart}
\usepackage{amsfonts}
\usepackage{amsmath,amscd}
\usepackage{amssymb}
\usepackage{amsthm}
\usepackage{newlfont}
\newcommand{\f}{\frac}

\long\def\alert#1{\parindent2em\smallskip\hbox to\hsize
{\hskip\parindent\vrule
\vbox{\advance\hsize-2\parindent\hrule\smallskip\parindent.4\parindent
\narrower\noindent#1\smallskip\hrule}\vrule\hfill}\smallskip\parindent0pt}
 \newtheorem{thm}{Theorem}[section]
\newtheorem{cor}[thm]{Corollary}
 \newtheorem{lem}[thm]{Lemma}
 \newtheorem{prop}[thm]{Proposition}
\theoremstyle{definition}
 \newtheorem{defn}[thm]{Definition}
\theoremstyle{remark}

 \numberwithin{equation}{section}

\begin{document}

\title[ the direct sum of Lie algebras ]
 {$c$-nilpotent multiplier and $c$-capability of the direct sum of Lie algebras}
\author[F. Johari]{Farangis Johari}
\author[P. Niroomand]{Peyman Niroomand}
\author[M. Parvizi]{Mohsen Parvizi}

\address{Department of Pure Mathematics\\
Ferdowsi University of Mashhad, Mashhad, Iran}
\email{e-mail:farangis.johari@mail.um.ac.ir,farangisjohary@yahoo.com}
\address{School of Mathematics and Computer Science\\
Damghan University, Damghan, Iran}
\email{niroomand@du.ac.ir, p$\_$niroomand@yahoo.com}

\address{Department of Pure Mathematics\\
Ferdowsi University of Mashhad, Mashhad, Iran}
\email{parvizi@um.ac.ir}



\date{\today}


\begin{abstract}
In this paper, we determine the behavior of the $c$-nilpotent multiplier of Lie algebras with respect to the direct sums. Then we give some results on the $c$-capability of the direct sum of finite dimensional Lie algebras.
\end{abstract}

\maketitle

\section{Motivation and Introduction }
Let $L$ be a Lie algebra presented as the quotient algebra of a free Lie algebra $ F $ by an ideal $R.$ Then the $c$-nilpotent multiplier of $L,$ is defined to be
\[\mathcal{M}^{(c)}(L)=\frac{ R\cap F^{c+1}}{[R,_c F]}, ~\text{for all}~c\geq 1\]
where $F^{c+1} $ is the $(c + 1)$-th term of the lower central series of $F$ and $[R,_0 F]= R,
[R,_c F]=[[R,_{c-1} F], F].$ This is analogous to the definition of the Baer-invariant of a group with respect to the
variety of nilpotent groups of class at most $c$ given by Baer in \cite{1b} (see \cite{28b,4b,el5,11b,sal2} for more information
on the Baer invariant of groups). The $1$-nilpotent multiplier of $L ,$ is the more studied as the Schur multiplier of $ L,$ $\mathcal{M}(L)=R\cap F^2/[R,F],  $  (see for instance \cite{24b,2b,3b,34b, 20b}).
 It is proved that the Lie algebra $\mathcal{M}^{(c)}(L)$ is abelian and independent
of the choice of the free Lie algebra $F.$

Looking \cite{24b,2b,3b,28b,34b,33b,21b} show that the behavior of the Schur multiplier with respect to direct
sum of two Lie algebras, may leads us to have more results on the Schur multiplier of a Lie algebra.\\
From \cite{miller}, the  formula of  the Schur multiplier for the direct product of two groups  is well known. Later, the same result for the direct sum  of two Lie algebras  was proved  in \cite{3b,salco}.  Moghaddam  in \cite{11b} extended  this result for the c-nilpotent multiplier of the direct product of two groups and also Ellis  improved   the result  of Moghaddam  in \cite{el5}.
 The last two authors in \cite{21b} showed the behaviour of the $2$-nilpotent
multipliers respect to the direct sum of two Lie algebras. Recently, Salemkar and Aslizadeh obtained  a formula for the $c$-nilpotent multipliers of the direct sum of Lie algebras whose abelianizations are finite dimensional (see \cite[Theorem 2.5]{salj}) and also generalize it for arbitrary Lie algebras, in the case $c+1$ is a prime number or $c+1=4$ which is a strong restriction (see \cite[Theorem 2.9]{salj}). Here, we intend to generalize the result of Salemkar et. al. to the $c$-nilpotent multipliers for any arbitrary $c$, and then we give some results concerning the $c$-capability of the direct sum of Lie algebras.
\section{The $c$-th  term of the lower central series of free product of two Lie algebras}
In this section, we are going to obtain  the formula  of  the $c$-th  term of lower central series of free product of two Lie algebras.  \\
The definition of  basic commutators  plays a fundamental role in obtaining our main results.

Following Shirshov \cite{22b} for a free Lie algebra $L$ on the set $X=\{x_1,x_2,\ldots \}$.
The  basic commutators on the set $X$ defined inductively as follows.
\begin{itemize}
\item[(i)] The generators $x_1,x_2,\ldots, x_n$ are basic commutators of length one and ordered by setting $x_i < x_j$ if $i < j.$

\item[(ii)] If all the basic commutators $d_i$ of length less than $t$ have been defined and ordered, then we may define the basic commutators of length $t$ to be all commutators of the form $[d_i, d_j]$ such that the sum of lengths of $d_i$ and $d_j$ is $t$, $d_i > d_j$, and if $d_i =[d_s, d_t]$, then $d_j\geq d_t$. The basic commutators of length $t$ follow those of lengths less than $t$. The basic commutators of the same length can be ordered in any way, but usually the lexicographical order is used.
\end{itemize}

   The number of all basic commutators on a set $X=\{x_1,x_2,\ldots x_d\}$ of length $n$ is denoted by $l_d(n)$. Thanks to \cite{22b},  we have
   \[l_d(n)=\frac{1}{n}\sum_{m|n}\mu (m)d^{\f{n}{m}},\]
   where $\mu (m)$ is the M\"{o}bius function, defined by $\mu (1) = 1, \mu (k) = 0$ if $k$ is divisible by a square, and
$\mu (p_1 \ldots p_s) = (-1)^s $ if $p_1,\ldots , p_s$ are distinct prime numbers.
Using the topside statement and looking  \cite[Lemma 1.1]{20b} and \cite{22b}, we have the next theorem.
\begin{thm}\label{13}
Let $ F $ be a free Lie algebra on a set $ X $, then $ F^c/ F^{c+i}$ is an abelian Lie algebra with the basis of all basic commutators on $ X $ of lengths $ c,c+1,\ldots,c-i+1 $ for all $i$, $0 \leq i \leq c$. In particular, $ F^c/ F^{c+1}$ is an abelian Lie algebra of dimension $l_d(c)$.
\end{thm}
The following definition is vital and will be used in the rest.
\begin{defn}\label{3}
Let $ c\geq 0.$
Consider the free product $ A*B $ of two Lie algebras $A$ and $B.$
Let us impose the ordering $ A<B.$ The set of  basic commutators of weight $ c $ on two letters  $ A$ and $B $  is denoted by $ M.$ We define
\[\sum (A*B)_{c}=\langle [P_1,\ldots,P_{c}]_{\lambda}|\lambda\in M\rangle,\]
where $ P_i=A $ or $ P_i=B.$ In fact, $\sum (A*B)_{c}  $ is the  subalgebra generated by all the basic commutator subalgebras $[P_1,\ldots,P_{c}]_{\lambda}  $ such that $ \lambda\in M. $
\end{defn}
For example, we have $\sum (A*B)_{1}=A*B,$ $\sum (A*B)_{2}=[A,B]  $ and $ \sum (A*B)_{3}=\langle [B,A,A], [B,A,B]\rangle.$
\begin{lem}\label{4}
Let $A$ and $B$ be two Lie algebras. Then $\sum (A*B)_{c}=\langle [\sum (A*B)_{c-1},A],[\sum (A*B)_{c-1},B] \rangle =[\sum (A*B)_{c-1},A*B]  $  and
$\sum (A*B)_{c}  $ is an ideal of $ A*B$ and $ c\geq 2.$
\end{lem}
\begin{proof}
Clearly, $\sum (A*B)_{c+1} =\langle [\sum (A*B)_{c},A],[\sum (A*B)_{c},B] \rangle.$ We proceed by induction on $c.$ If $c=2, $ then $\sum (A*B)_{2}=[A,B].  $  Let $c\geq 3.$
By the induction hypothesis, we have  $\sum (A*B)_{c}  $ is an ideal of $ A*B$ and
\[\sum (A*B)_{c}=\langle [\sum (A*B)_{c-1},A],[\sum (A*B)_{c-1},B] \rangle =[\sum (A*B)_{c-1},A*B]. \]
We claim that \[ \langle [\sum (A*B)_{c},A],[\sum (A*B)_{c},B] \rangle =[\sum (A*B)_{c},A*B]. \] Clearly, $ \langle [\sum (A*B)_{c},A],[\sum (A*B)_{c},B] \rangle \subseteq [\sum (A*B)_{c},A*B]. $ It is enough to show that \[ [\sum (A*B)_{c},A*B] \subseteq   \langle [\sum (A*B)_{c},A],[\sum (A*B)_{c},B] \rangle . \]
Let $l=a+b+w  \in A*B, $ $ w=\sum_{i=1}^n [a_i,b_i] $ and $ x\in \sum (A*B)_{c}$ such that $ a,a_i\in A, $ $ b,b_i\in B $ and $ w\in [A,B].$
 We show that $ [x,l]\in \langle [\sum (A*B)_{c},A],[\sum (A*B)_{c},B] \rangle. $\\
For this, we know $ [x,l]=[x,a+b+w]=[x,a]+[x,b]+[x,w].$ Clearly, \[ [x,a]+[x,b]\in \langle [\sum (A*B)_{c},A],[\sum (A*B)_{c},B] \rangle. \] Since
$[x,w]=[x,\sum_{i=1}^n [a_i,b_i]]=\sum_{i=1}^n [x, [a_i,b_i]],$
it is enough to see that $  [x, [a_i,b_i]]\in  \langle [\sum (A*B)_{c},A],[\sum (A*B)_{c},B] \rangle. $
 The induction hypothesis implies $ [b_i,x], [x,a_i]\in \sum (A*B)_{c}.$  By the Jacobian identity, $[x, [a_i,b_i]]=[b_i,x,a_i]+ [x,a_i,b_i] $ and so \[ [x, [a_i,b_i]]=[b_i,x,a_i]+ [x,a_i,b_i]\in \langle [\sum (A*B)_{c},A],[\sum (A*B)_{c},B] \rangle. \]
Hence $[x,w]\in   \langle [\sum (A*B)_{c},A],[\sum (A*B)_{c},B] \rangle. $ Therefore  \[\sum (A*B)_{c+1} = [\sum (A*B)_{c},A*B].\] Now, we show that  $\sum (A*B)_{c+1}  $ is an ideal of $ A*B.$ Let
$y\in A*B  $ and $x=\sum_{i=1}^n [x_i,l_i]\in  \sum (A*B)_{c+1}  $
such that $ x_i\in \sum (A*B)_{c}  $ and $ l_i\in A*B.  $
Since $ [x,y] =\sum_{i=1}^n [x_i,l_i,y]$ and $  [x_i,l_i]\in \sum (A*B)_{c} $ for all $ 1\leq i\leq n, $ we conclude $ [x,y]\in   \sum (A*B)_{c+1}. $
Hence $\sum (A*B)_c$ is an ideal and result follows.
\end{proof}
We give the structures of the $(c + 1)$-th term of the the lower central series of a free product of two Lie algebras.
\begin{prop}\label{5}
Let $A$ and $B$ be two Lie algebras. Then
\[(A*B)^{c}=A^{c}+B^{c}+\sum (A*B)_{c}\] for $ c\geq 1. $
\end{prop}
\begin{proof}
We proceed by induction on $ c.$ If $ c=1,$ then $ A*B=A+B+ [A,B]$ and the result holds. Let $ c\geq 2.$ By the induction hypothesis , $ (A*B)^{c}=A^{c}+B^{c}+\sum (A*B)_{c}. $ It is easy to see that
\begin{align*}&(A*B)^{c+1}=[(A*B)^{c},A*B]=[A^{c}+B^{c}+\sum (A*B)_{c},A*B]\\&=[A^{c},A*B]+[B^c,A*B]+[\sum (A*B)_{c},A*B].\end{align*}
  By Lemma \ref{4}, $ [\sum (A*B)_{c},A*B]=\sum (A*B)_{c+1}. $
 We conclude that \[ A^{c+1}+B^{c+1}+\sum (A*B)_{c+1}\subseteq (A*B)^{c+1}=[A^{c},A*B]+[B^c,A*B]+\sum (A*B)_{c+1}. \]
  Now, we show that \[[A^{c},A*B]+[B^c,A*B]+\sum (A*B)_{c+1}\subseteq  A^{c+1}+B^{c+1}+\sum (A*B)_{c+1}.\]
Let $ x\in A^c,y\in B^c $ and $l=a+b+w  \in A*B, $ $ w=\sum_{i=1}^n [a_i,b_i] $  such that $ a,a_i\in A, $ $ b,b_i\in B $ and $ w\in [A,B].$ Since
\[ [x,l]=[x,a]+[x,b]+[x,w] ~\text{and}~ [y,l]=[y,a]+[y,b]+[y,w],\] we can see that $ [x,b]+[x,w],[y,a]+[y,w]\in  \sum (A*B)_{c+1}.$
Hence $ [A^{c},A*B]+[B^c,A*B]+\sum (A*B)_{c+1}= A^{c+1}+B^{c+1}+\sum (A*B)_{c+1}. $
It completes the proof.
\end{proof}
The next step gives  a generating set for $\sum (F_1*F_2)_{c+1}  $ in terms of the
free generators of $F_1$ and $F_2.$
\begin{prop}\label{7}
Let $F_1$ and $F_2$ be two free Lie algebras  generated by  the set $X$ and $Y,$ respectively,
and $F = F_1 * F_2.$  Then $\sum (F_1*F_2)_{c+1}+ F^{c+2}/  F^{c+2}$ is an abelian Lie
algebra with the basis of all basic commutators $ \lambda $ of weight $ c+1 $ in the set $X\cup Y $ which involve at least one $ x_i $ and at least one $ y_j.$
\end{prop}
\begin{proof}
By Proposition \ref{5}, we have\[
F^{c+1} / F^{c+2}=(F_1^{c+1}+F_2 ^{c+1} +\sum (F_1*F_2)_{c+1}) / F^{c+2}\cong  (F_1^{c+1} / F_1^{c+2}) \oplus \]\[(F_2^{c+1} / F_2^{c+2})\oplus (\sum (F_1*F_2)_{c+1}+ F^{c+2} / F^{c+2}).\]
By Theorem \ref{13}, $F^{c+1} / F^{c+2}$  $,F_1^{c+1} / F_1^{c+2}$ and $F_2^{c+1} / F_2^{c+2}$ are  the abelian Lie algebras with the basis of all basic commutators of weight $ c+1 $ on $X\cup Y,X$ and $Y,$ respectively. We conclude that $ (\sum (F_1*F_2)_{c+1}+F^{c+2}) / F^{c+2} $  is an abelian Lie
algebra with the basis of all basic commutators $ \lambda $ of weight $ c $ in the $X\cup Y $ which involve at least one $ x_i $ and at least one $ y_i, $ as required.
\end{proof}
\begin{cor}\label{81}
Let $F_1$ and $F_2$ be two free Lie algebras  generated by  $X$ and $Y$ with $ d_1 $ and $ d_2 $ elements,  respectively,
and $F = F_1 * F_2.$  Let $ S $ be the set of all basic commutators $ \lambda $ of weight $ c $ in the $X\cup Y $ which involve at least one $ x $ and at least one $ y. $ Suppose that the order is defined as $ x_i<x_j<y_t<y_d  $ for $ i<j $ and $ t<d, $ where
$x_i, $   $x_j\in X  $ and $y_t, y_d\in Y. $
Then  $ \dim \langle S\rangle =l_{d_1+d_2}(c)-l_{d_1}(c)-l_{d_2}(c).$
\end{cor}
\begin{proof}
$ S $ is the set of all basic commutators $ \lambda $ of weight $ c $ in the $X\cup Y $ such that $ \lambda $ is not basic commutator of weight $ c $ on the set $X$ or $ Y. $ By Theorem \ref{13}, $ \dim \langle S\rangle =l_{d_1+d_2}(c)-l_{d_1}(c)-l_{d_2}(c).$ The result follows.
\end{proof}

\begin{cor}\label{8}
Let $F_1$ and $F_2$ be two free Lie algebras  generated by the set $X$ and $Y,$ respectively,
and $F = F_1 * F_2.$  Let $ S $ be the set of all basic commutators $ \lambda $ of weight $ c+1 $ on $X\cup Y $ which involve at least one $ x_i $ and at least one $ y_i. $ Let we have the order
$ x_i<x_j<y_t<y_d  $ for $ i<j $ and $ t<d, $ where
$x_i, $   $x_j\in X  $ and $y_t, y_d\in Y. $
Then  $\sum (F_1*F_2)_{c+1}=\sum (F_1*F_2)_{c+2}+\langle S\rangle.$
\end{cor}
\begin{proof}
By Proposition \ref{7}, we have $(\sum (F_1*F_2)_{c+1}+ F^{c+2})/  F^{c+2}=(\langle S\rangle + F^{c+2})/  F^{c+2}$ and so $ \sum (F_1*F_2)_{c+1}+ F^{c+2}= \langle S\rangle + F^{c+2}.$ Thus \[\sum (F_1*F_2)_{c+1}= (\langle S\rangle + F^{c+2})\cap \sum (F_1*F_2)_{c+1}= (\sum (F_1*F_2)_{c+1} \cap  F^{c+2})+\langle S\rangle.\]
By Propositions \ref{5}, $\sum (F_1*F_2)_{c+1}= \Big{(}\big{(}F_1^{c+2}+F_2^{c+2}+ \sum (F_1*F_2)_{c+2}\big{)}\cap \sum (F_1*F_2)_{c+1} \Big{)} +\langle S\rangle=\Big{(}\big{(}F_1^{c+2}+F_2^{c+2} \big{)}\cap  \sum (F_1*F_2)_{c+1} \Big{)} + \sum (F_1*F_2)_{c+2}  +\langle S\rangle,$ as required.
\end{proof}

\section{The $c$-nilpotent multiplier of a direct sum of Lie Algebras}

 In this section, we study the $c$-nilpotent multiplier with respect to the direct sum of two Lie algebras, and then we give some results on the $c$-capability of the direct sum of two Lie algebras. The techniques are used here are based on the notion of free products of Lie algebras
and the expansion of an element of a free Lie algebra in terms of basic commutators. Every
element of a free Lie algebra can be expressed as a sum of some basic commutators. For the
elements of a free product of two free Lie algebras also we have a similar expression except
that the basic commutators are on the union of the bases of the two free Lie algebras taken
to form the free product. It is easy to see that the free product of two free Lie algebras is
actually a free Lie algebra on the disjoint union of the bases of the two chosen free Lie algebras (see \cite{22b} for more details). \\
Let $ L_1 $ and $ L_2 $ be two Lie algebras with  the following free presentations
\[0\rightarrow R_1 \rightarrow F_1 \rightarrow L_1\rightarrow 0~ \text{and}~ 0\rightarrow R_2 \rightarrow F_2 \rightarrow L_2\rightarrow 0,\]
respectively.
Then the free presentation of the direct sum  $ L_1\oplus L_2 $ is given in the following.
\begin{lem}\cite[Lemma 2.1]{21b}\label{1}
Let $ F=F_1*F_2 $ be the free product of two free Lie algebras $ F_1 $ and $ F_2.$ Then
$0\rightarrow R \rightarrow F \rightarrow L_1 \oplus L_2\rightarrow 0$
is the free presentation for $ L_1 \oplus L_2 $ in which $ R=R_1+R_2+[F_2,F_1]. $
\end{lem}
The following lemma plays a key role in the main result, also it is different from \cite[Proposition 2.1]{salj}, since we use the concept of free products in order to obtain the free presentation of a direct sum of Lie algebras.\\
By applying Lemma \ref{1} and the above notation, we can compute the $c$-nilpotent multiplier of $ L_1 \oplus L_2 $ in terms
of $F_i $'s and $R_i$'s as follows
\[\mathcal{M}^{(c)}(L_1\oplus L_2)=\frac{ R\cap F^{c+1}}{[R,_c F]}=\frac{(R_1+R_2+[F_2,F_1] \cap (F_1*F_2)^{c+1})}{[R_1+R_2+[F_2,F_1],_c F_1*F_2]}.\]
Define  \begin{equation}\label{eq}
\eta:
\mathcal{M}^{(c)}(L_1\oplus L_2)=\frac{ R\cap F^{c+1}}{[R,_c F]}\rightarrow \frac{ R_1\cap F_1^{c+1}}{[R_1,_c F_1]}\oplus
\frac{ R_2\cap F_2^{c+1}}{[R_2,_c F_2]}=\mathcal{M}^{(c)}(L_1)\oplus \mathcal{M}^{(c)}( L_2).
\end{equation} which is induced by the canonical homomorphism
from $ F=F_1*F_2\rightarrow F_1\times F_2.$  Then we have
\begin{lem}\label{2}
Let $ L_1 $ and $ L_2 $ be two Lie algebras. Then
\[\mathcal{M}^{(c)}(L_1\oplus L_2)\cong \mathcal{M}^{(c)}(L_1) \oplus \mathcal{M}^{(c)}( L_2)\oplus K,\]
where   $K=\mathrm{ker} \eta.$
\end{lem}
\begin{proof}
Let $ F=F_1*F_2.$ Then the epimorphism $ F\rightarrow F_1\times F_2 $ induces the above epimorphism $ \eta.$
Consider the map
\[\beta:\frac{ R_1\cap F_1^{c+1}}{[R_1,_c F_1]}\oplus \frac{ R_2\cap F_2^{c+1}}{[R_2,_c F_2]}\rightarrow \frac{ R\cap F^{c+1}}{[R,_c F]} \]
defined by $ (x_1+[R_1,_c F_1], x_2+[R_2,_c F_2]) \mapsto x_1+x_2+[R,_c F].$ Clearly, $ \beta $ is a well-defined homomorphism. It is easy to see that $ \beta $ is a left inverse to $\eta$ in \ref{eq}. Therefore the sequence
\[0\rightarrow K \rightarrow \mathcal{M}^{(c)}(L_1\oplus L_2) \rightarrow \mathcal{M}^{(c)}(L_1) \oplus \mathcal{M}^{(c)}( L_2) \rightarrow 0\]
splits and the result holds.
\end{proof}
We recall that
Now we compute the kernel of the epimorphism $ \eta $ in \ref{eq}.
\begin{thm}\label{6}
Let
\[\eta:\mathcal{M}^{(c)}(L_1\oplus L_2)\rightarrow \mathcal{M}^{(c)}(L_1)\oplus \mathcal{M}^{(c)}(L_2)\]
be the epimorphism defined in  \ref{eq}. Then
\[ \ker \eta= \sum (F_1*F_2)_{c+1}+[R,_cF]/[R,_c F].\]
\end{thm}
\begin{proof}
Clearly
$(\sum (F_1*F_2)_{c+1}+[R,_c F])/[R,_c F]\subseteq \ker \eta.$
Let $  w+ [R,_c F]\in \ker \eta.  $ Using Proposition \ref{5}, we have  $ w+ [R,_c F]=a+b+z+[R,_c F]\in \ker \eta $ such that
$ a\in   F_1^{c+1},$ $ b\in F_2^{c+1} $ and
$z\in \sum (F_1*F_2)_{c+1}.$ The definition of $ \eta  $ implies
$a\in  [R_1,_c F_1]$ and $ b\in  [R_2,_c F_2]$ so that $ w+[R,_c F] =z+[R,_c F],$ as required.
\end{proof}
The following corollary is an immediate consequence of Lemma \ref{2} and Theorem \ref{6}.
\begin{cor}\label{61}
Let $ L_1 $ and $ L_2 $ be two Lie algebras. Then
\[\mathcal{M}^{(c)}(L_1\oplus L_2)\cong \mathcal{M}^{(c)}(L_1) \oplus \mathcal{M}^{(c)}( L_2)\oplus \big{( }\sum(F_1*F_2)_{c+1}+[R,_cF]\big{)}/[R,_c F].\]
\end{cor}

\begin{lem}\label{87}
Let $F_1$ and $F_2$ be two free Lie algebras  generated by the set $X$ and $Y,$ respectively,
and $F = F_1 * F_2.$  Let $ S $ be the set of all basic commutators $ \lambda $ of weight $ c+1 $ on $X\cup Y $ which involve at least one $ x_i $ and at least one $ y_i. $ Let we have the order
$ x_i<x_j<y_t<y_d  $ for $ i<j $ and $ t<d, $ where
$x_i, $   $x_j\in X  $ and $y_t, y_d\in Y. $
Then   $\sum (F_1*F_2)_{c+1}+[R,_cF]/[R,_cF]=\langle S\rangle+[R,_cF]/[R,_cF].$
\end{lem}
\begin{proof}
By Corollary \ref{8}, we have $\sum (F_1*F_2)_{c+1}=\sum (F_1*F_2)_{c+2}+\langle S\rangle.$
Also by Lemma \ref{4}, we have $\sum (F_1*F_2)_{c+2}=[\sum (F_1*F_2)_{c+1},F].  $  Since $ [R,_cF]=[R_1+R_2+[F_2,F_1],_c F], $  so $ \sum (F_1*F_2)_{c+2}=[\sum (F_1*F_2)_{c+1},F]\subseteq [[F_2,F_1],_c F] \subseteq [R,_cF]=[R_1+R_2+[F_2,F_1],_c F].$
The result follows.
\end{proof}
Similar to the following  definition is found in \cite{el5}.
\begin{defn}\label{9}
Let $ c\geq 1, K$ and $ H $ be  two abelian  Lie algebras with bases $ \{a_i\}_{i\in I}  $ and   $ \{b_j\}_{j\in j},$ respectively. Let we have the ordering $ K<H$ and the set of  basic commutators of weight $ c $ on the letters  $ K$ and $H$ which contains at least one $H$ and one $K$, is denoted by $S_1.$ We define
\begin{equation}\label{e1}
 \tau (K,H)_{c}=\oplus_{\lambda \in S_1 }\Big{(}P_1\otimes \ldots\otimes P_{c}\Big{)},
 \end{equation}
where $ P_i=K $ or $ P_i=H.$
\end{defn}
Note that  descriptions of $\tau (K,H)_{c}  $ follows from basic properties of the tensor
products of abelian Lie algebras and the definition of basic commutators.
\begin{align*}
\tau(K, H)_{1}&=0,\\
\tau(K, H)_{2}&=(H\otimes K),\\
\tau(K, H)_{3}&=(H\otimes K\otimes  K)\oplus( H\otimes K\otimes H),\\
\tau(K, H)_{4}&=(H\otimes K\otimes  K\otimes K)\oplus( H\otimes K\otimes K\otimes H)\oplus (H\otimes K\otimes  H\otimes H).
\end{align*}
Multi linearity of the generating elements of $\sum (F_1,F_2)_{c+1} (\mod ~ [R,_{c}F])$ imposes
a connection between  $\sum (F_1,F_2)_{c+1} (mod ~ [R,_{c}F])$ and
$ \tau (L_1^{ab},L_2^{ab})_{c+1}.$
\begin{lem}\label{11}
With the previous notations and assumptions in Proposition \ref{7}, we have
$\tau (F_1^{ab},F_2^{ab})_{c+1}\cong  \sum (F_1,F_2)_{c+1}+[R,_cF]/[R,_c F]$
for $ c\geq 1.$
\end{lem}
\begin{proof}
By Lemma \ref{87}, we have
obviously, the map $ \alpha_1: \sum (F_1,F_2)_{c+1}/[R,_c F]\rightarrow \tau (L_1^{ab},L_2^{ab})_{c+1}$ given by
$ [f_1,\ldots,f_{c+1}]\mapsto  \bar{f_1}\otimes \ldots\bar{f_{c+1}} $ is a Lie homomorphism. Conversely, we may check that
$ \alpha_2: \tau (L_1^{ab},L_2^{ab})_{c+1}\rightarrow \sum (F_1,F_2)_{c+1}/[R,_c F]$ given by
$  \bar{f_1}\otimes \ldots\otimes \bar{f_{c+1}} \mapsto [f_1,\ldots,f_{c+1}]$ is a Lie homomorphism too.
Now $ \eta_1\eta_2$ and $ \eta_2\eta_1 $ are identity, so the result follows.
\end{proof}
The following theorem generalized \cite[Theorems 2.5 and 2.9]{salj} and states a formula for the $c$-nilpotent multiplier of a direct sum of two arbitrary Lie algebras without any restriction on $c$.
\begin{thm}\label{12}
Let $L_1$ and $L_2$ be arbitrary Lie algebras. Then
\[\mathcal{M}^{(c)}(L_1\oplus L_2)\cong  \mathcal{M}^{(c)}(L_1)\oplus \mathcal{M}^{(c)}(L_2) \oplus \tau(L_1^{ab},L_2^{ab})_{c+1}.\]
\end{thm}
\begin{proof}
The result  follows immediately from  Corollary \ref{61} and Lemma \ref{11}.
\end{proof}
\section{The $c$-capability of a direct sum of finite dimensional Lie Algebras}
In this section, we are going to determine the $c$-capability of all abelian  Lie algebras. Then we discuss on the $c$-capability  of a direct sum of a finite dimensional non-abelian Lie algebra and an abelian Lie algebra.

Recall that from \cite{sal2} a Lie algebra $ L$ is $c$-capable if there exists some Lie algebra $ H $ such that $L\cong H/Z_c(H),$ where $ Z_c(H)$ is the $c$-th center of $K.$
   Evidently, $L$ is $1$-capable if and only if it is an inner derivation Lie
algebra, and $L$ is $c$-capable ($c\geq 2$) if and only if it is an inner derivation Lie algebra of a $(c-1)$-capable Lie algebra.
\\
In  \cite{sal2}, the $c$-epicenter of a Lie algebra $L, $    $Z^*_c (L),$  is defined to be the smallest ideal $M$ of $L$ such that $L/M$ is $c$-capable. For $c=1,$ the $1$-epicenter of $ L $ is equal to
 $Z^* (L)$ for a Lie algebra $L $ was defined in \cite{alam}.
It is obvious that $Z^*_c (L)$ is a characteristic ideal of $L$ contained in $Z_c(L),$ and $Z^*_c (L/Z^*_c(L))=0.$ So $ L $ is $c$-capable if and only if $Z_c^{*}(L)=0.$\\
The proof of the following lemma is similar to the proof of \cite[Theorem 2.7]{ni3}.
\begin{lem}\label{511}
Let $ A$ and $ B $ be Lie algebras. Then
$ Z_c^{*}(A\oplus B)\subseteq   Z_c^{*}(A)\oplus Z_c^{*}(B)$
for all $ c\geq 1.$
\end{lem}
\begin{proof}
Since $(A\oplus B)/(Z_c^{*}(A)\oplus Z_c^{*}(B))\cong (A/Z_c^{*}(A))\oplus (B/Z_c^{*}(B)),  $ we have $Z_c^{*}(A\oplus B)\subseteq   Z_c^{*}(A)\oplus Z_c^{*}(B),$ as required.
\end{proof}
The following results show that the capability of the direct product of a  non-abelian Lie algebra and an  abelian Lie algebra depends only on the capability of its non-abelian factor.

\begin{prop}\label{p}
Let $ L $ be a finite dimensional   Lie algebra. Then $ L\cong T\oplus A $ in which $ A $ is an abelian Lie algebra and $Z(L)\cap L^2=Z(T).$  Moreover, $Z_c^*(L)= Z_c^*(T) \subseteq T^2 $  for all $ c\geq 1.$
\end{prop}
\begin{proof}
By applying \cite[Proposition 3.1]{ni8}, we have $ L\cong T\oplus A $  such that
$Z(L)\cap L^2=Z( T) $ and $ A $ is an abelian Lie algebra. Moreover, $Z_c^*(L)\subseteq Z_c^*(T) \subseteq T^2 $  for all $ c\geq 1.$  If $Z_c^{*}(T)=0,  $ then $ Z_c^{*}(L)= Z_c^{*}(T)=0. $
Now let $ Z_c^{*}(T)\neq 0. $ We claim that  $ Z_c^{*}(T)\subseteq Z_c^{*}(L).$
By invoking Theorem \ref{12}, we have
\[ \mathcal{M}^{(c)}(L)\cong   \mathcal{M}^{(c)}(T)\oplus  \mathcal{M}^{(c)}(A) \oplus \tau(T/T^2, A)_{c+1}\]
and
\[ \mathcal{M}^{(c)}(L/Z_c^{*}(T))\cong   \mathcal{M}^{(c)}(T/Z_c^{*}(T))\oplus  \mathcal{M}^{(c)}(A) \oplus \tau(T/T^2, A)_{c+1}.\]
Now \cite[Corollary 2.4]{sal2} implies $\dim \mathcal{M}^{(c)}(T)=\dim \mathcal{M}^{(c)}(T/Z_c^{*}(T))-\dim (Z_c^{*}(T)\cap L^{c+1}).$ Thus \begin{align*} &\dim \mathcal{M}^{(c)}(L)=\dim \mathcal{M}^{(c)}(T/Z_c^{*}(T))+\dim  \mathcal{M}^{(c)}(A)+\dim  \tau(T/T^2, A)_{c+1}\\&-\dim Z_c^{*}(T)\cap L^{c+1}=\dim \mathcal{M}^{(c)}(L/Z_c^{*}(T))-\dim (Z_c^{*}(T)\cap L^{c+1}). \end{align*}
Again by \cite[Corollary 2.4]{sal2}, $ Z_c^{*}(T)\subseteq Z_c^{*}(L),$ as required.
\end{proof}
The following corollary is an immediate consequence of Proposition \ref{p}.
\begin{cor}\label{pp}
Let $ L=T\oplus A(n) $ be a finite dimensional Lie algebra such that
$ T $ is a non-abelian Lie algebra. Then $ L $ is $c$-capable if and only if $ T $ is $c$-capable.
\end{cor}
\begin{thm}
Let $ L=L_1\oplus  L_2 $ such that $ L_1 $ and $ L_2 $ are finite dimensional non-abelian Lie algebras. Then  $Z_c^{*}(L_1\oplus  L_2 )=Z_c^{*}(L_1)\oplus  Z_c^{*}(L_2).$
\end{thm}
\begin{proof}
 We have $ L_i=T_i\oplus A_i $ and $Z_c^{*}(T_i)= Z_c^{*}(L_i)  $ for $ 1\leq i\leq 2,$ by Proposition \ref{p}. Therefore $L=T_1\oplus  T_2 \oplus A $ and $Z_c^{*}(L)= Z_c^{*}(T_1\oplus  T_2),$ where $ A=A_1\oplus A_2, $ by Proposition \ref{p}.
We claim that $Z_c^{*}(T_1\oplus  T_2) =Z_c^{*}(T_1)\oplus  Z_c^{*}(T_2). $ Lemma \ref{511} implies $Z_c^{*}(T_1\oplus  T_2) \subseteq Z_c^{*}(T_1)\oplus  Z_c^{*}(T_2). $
Now we show that  $ Z_c^{*}(T_i)\subseteq Z_c^{*}(T_1\oplus  T_2)$ for $ i=1,2. $ If $Z_c^{*}(T_i)=0 $ for $ i=1,2, $ then $ Z_c^{*}(T)= Z_c^{*}(T_1)\oplus Z_c^{*}(T_2) =0. $
Now we have  $ Z_c^{*}(T_i)\neq 0 $ for $ i=1,2, $ or $ Z_c^{*}(T_1)\neq 0 $ and $ Z_c^{*}(T_2)=0, $ or  $ Z_c^{*}(T_2)\neq 0 $ and $ Z_c^{*}(T_1)=0. $
First consider $ Z_c^{*}(T_i)\neq 0 $ for $ i=1,2. $
By invoking Theorem \ref{12}, we have
\[ \mathcal{M}^{(c)}(T_1\oplus  T_2)\cong   \mathcal{M}^{(c)}(T)\oplus  \mathcal{M}^{(c)}(T_2) \oplus \tau(T_1/T_1^2, T_2/T_2^2)_{c+1},\]
\[ \mathcal{M}^{(c)}(T_1\oplus  T_2/Z_c^{*}(T_1))\cong   \mathcal{M}^{(c)}(T_1/Z_c^{*}(T_1))\oplus  \mathcal{M}(T_2) \oplus \tau(T_1/T_1^2, T_2/T_2^2)_{c+1}\]
and
\[ \mathcal{M}^{(c)}(T_1\oplus  T_2/Z_c^{*}(T_2))\cong   \mathcal{M}^{(c)}(T_1)\oplus  \mathcal{M}(T_2/Z_c^{*}(T_2)) \oplus \tau(T_1/T_1^2, T_2/T_2^2)_{c+1}.\]
Now \cite[Corollary 2.4]{sal2}  implies $\dim \mathcal{M}^{(c)}(T_1)=\dim \mathcal{M}^{(c)}(T_1/Z_c^{*}(T_1))-\dim Z_c^{*}(T_1)\cap (T_1)^{c+1}$
and $\dim \mathcal{M}^{(c)}( T_2)=\dim \mathcal{M}^{(c)}(T_2/Z_c^{*}(T_2))-\dim Z_c^{*}(T_1)\cap (T_2)^{c+1}.$
Thus $ \dim \mathcal{M}^{(c)}(T_1\oplus  T_2)= \dim \mathcal{M}^{(c)}(T_1/Z_c^{*}(T_1))-\dim Z_c^{*}(T_1)\cap (T_1)^{c+1}+\dim \mathcal{M}^{(c)}(T_2)+\tau(T_1/T_1^2, T_2/T_2^2)_{c+1}<  \mathcal{M}^{(c)}(T_1\oplus  T_2/Z_c^{*}(T_1))$
and
\begin{align*} &\dim \mathcal{M}^{(c)}(T_1\oplus  T_2)= \dim \mathcal{M}^{(c)}(T_2/Z_c^{*}(T_2))-\dim Z_c^{*}(T_2)\cap (T_2)^{c+1}+\\&\dim \mathcal{M}^{(c)}(T_1)+\tau(T_1/T_1^2, T_2/T_2^2)_{c+1}<  \mathcal{M}^{(c)}(T_1\oplus  T_2/Z_c^{*}(T_2)).\end{align*} By using  \cite[Corollary 2.4]{sal2}, we have $ Z_c^{*}(T_i)\subseteq Z_c^{*}(T),$ for $ i=1,2. $ Thus $Z_c^{*}(L)= Z_c^{*}(T_1\oplus  T_2)  =Z_c^{*}(T_1)\oplus  Z_c^{*}(T_2)=Z_c^{*}(L_1)\oplus  Z_c^{*}(L_2).$ The proof is completed.
\end{proof}
Now we can state the following corollary which is interesting enough to be considered when studying the $c$-capability of finite dimensional non-abelian Lie algebras.

\begin{cor}
Let $L_1$ and $L_2$ be two finite dimensional non-abelian  Lie algebras and $c\geq 1$. Then $L_1\oplus L_2$ is $c$-capable if and only if $L_1$ and $L_2$ are $c$-capable.
\end{cor}


\begin{thebibliography}{99}%
{\small
\bibitem{1b}
R. Baer, Representations of groups as quotient groups, I, II, III, Trans. Amer. Math. Soc. 58 (1945) 295-419.
\bibitem{24b}
 P. Batten, Multipliers and covers of Lie algebras, Dissertation. State University, North Carolina (1993).
\bibitem{2b} P. Batten, E. Stitzinger, On covers of Lie algebras, Comm. Algebra 24 (1996) 4301-4317.
\bibitem{3b} P. Batten, K. Moneyhun, E. Stitzinger, On characterizing Lie algebras by their multipliers, Comm. Algebra 24 (1996) 4319-4330.
\bibitem{28b}J. Burns, G. Ellis, On the nilpotent multipliers of a group. Math. Z. 226, (1997) 405-428.
\bibitem{4b}J. Burns, G. Ellis, Inequalities for Baer invariants of finite groups, Canad. Math. Bull. 41 (4) (1998) 385-391.
\bibitem{el5}G. Ellis, On groups with a finite nilpotent upper central quotient, Arch. Math. 70 (1998) 89-96.
\bibitem{ni8}
F. Johari, M. Parvizi, P. Niroomand, Capability and Schur multiplier of a pair of Lie algebras,
J. Geometry Phys 114 (2017), 184-196.
\bibitem{10b} E. I. Marshall, The Frattini subalgebra of a Lie algebra, J. London Math. Soc. 42 (1967) 416-422.
\bibitem{11b} M. R. R. Moghaddam, The Baer-invariant of direct product, Arch. Math. (Basel) 33 (1980) 504-511.
\bibitem{miller}C. Miller, The second homology group of a group relations among commutators, Proc. Amer. Math. Soc. 3 (1952) 588-595.
\bibitem{34b} P. Niroomand, F. G. Russo,  A note on the Schur multiplier of a nilpotent Lie algebra. Commun. Algebra
39, (2011) 1293-1297.
\bibitem{ni3} P. Niroomand, M. Parvizi, F. G. Russo, Some criteria for detecting capable Lie algebras J. Algebra 384 (2013) 36-44.
\bibitem{33b} P. Niroomand, M. Parvizi, On the $2$-nilpotent multiplier of finite $p$-groups. Glasg. Math. J. 57(1), (2015) 201-210. 
\bibitem{21b}
P. Niroomand, M. Parvizi, $2$-nilpotent multipliers of a direct product of Lie algebras, Rend. Circ. Mat. Palermo 65 (2016) 519-523.
\bibitem{alam}
 A. R. Salemkar, V. Alamian, H. Mohammadzadeh, Some properties of the Schur multiplier and covers of Lie Algebras, Comm.
Algebra 36 (2008) 697-707.
\bibitem{20b}
A. R. Salemkar, B. Edalatzadeh, M. Araskhan, Some inequalities for the dimension of the $c$-nilpotent multiplier of Lie algebras,
J. Algebra 322 (2009) 1575-1585.
\bibitem{sal2}
A. Salemkar, Z. Riyahi, Some properties of the $c$-nilpotent multiplier of Lie algebras, J.Algebra 370 (2012) 320-325.
\bibitem{salco}A. R. Salemkar, B. Edalatzadeh, The multiplier and the cover of direct sums of Lie algebras, Asian-Eur. J. Math. 5 (2012) 1250026.
\bibitem{salj}
A. Salemkar,  A. Aslizadeh, The nilpotent multipliers of the direct sum of Lie algebras, J. of Algebra, 495 (2018) 220-232.
\bibitem{22b}
 A. I. Shirshov, On the bases of a free Lie algebra (Russian), Algebra i Logika Sem. 1 (1962), no. 1, 14-19.
}
\end{thebibliography}
\end{document}